\documentclass{amsart}

\usepackage{amssymb}
\usepackage{amsthm}
\usepackage{amsmath}
\usepackage{amsbsy}
\usepackage[all]{xy}
\usepackage{bm}
\usepackage{hyperref}
\usepackage{tikz}
\usepackage{array}
\usepackage{colortbl,xcolor}
\usepackage{ytableau}
\usepackage{hhline}
\usepackage{mathrsfs}
\usepackage{enumitem}
\allowdisplaybreaks

\newtheorem{theorem}{Theorem}[section]
\newtheorem{prop}[theorem]{Proposition}
\newtheorem{corollary}[theorem]{Corollary}
\newtheorem{question}[theorem]{Question}

\newtheorem{conjecture}[theorem]{Conjecture}
\newtheorem{lemma}[theorem]{Lemma}

\newcommand{\Abs}[1]{\left\lVert #1 \right\rVert}

\DeclareMathOperator{\ML}{ML}

\title[]{Barely lonely runners and very lonely runners}

\keywords{Lonely Runner Problem, Diophantine approximation.}
\subjclass[2010]{11K60 (primary), 11J13, 11J71, 52C07.}

\begin{document}

\author[]{Noah Kravitz}
\address[]{Grace Hopper College, Yale University, New Haven, CT 06510, USA}
\email{noah.kravitz@yale.edu}

\begin{abstract}

We introduce a sharpened version of the well-known Lonely Runner Conjecture of Wills and Cusick.  Given a real number $x$, let $\Vert x \Vert$ denote the distance from $x$ to the nearest integer.  For each set of positive integer speeds $v_1, \ldots, v_n$, we define the associated maximum loneliness to be
$$\ML(v_1, \ldots, v_n)=\max_{t \in \mathbb{R}}\min_{1 \leq i \leq n} \Vert tv_i \Vert.$$
The Lonely Runner Conjecture asserts that
$$\ML(v_1, \ldots, v_n) \geq \frac{1}{n+1}$$
for all choices of $v_1, \ldots, v_n$.  If the Lonely Runner Conjecture is true, then the quantity $1/(n+1)$ is the best possible, for there are known equality cases with $\ML(v_1, \ldots, v_n)=1/(n+1)$.  A natural but (to our knowledge) hitherto unasked question is: 
\begin{quote}
If $v_1, \ldots, v_n$ satisfy the Lonely Runner Conjecture but are not an equality case, must $\ML(v_1, \ldots, v_n)$ be uniformly bounded away from $1/(n+1)$?
\end{quote}
We conjecture that, contrary to what one might expect, this question has an affirmative answer that reflects an underlying rigidity of the problem.  More precisely, we conjecture that for each choice of $v_1, \ldots, v_n$, we have either $\ML(v_1, \ldots, v_n)=s/(ns+1)$ for some $s \in \mathbb{N}$ or $\ML(v_1, \ldots, v_n) \geq 1/n$.  Our main results are: confirming this stronger conjecture for $n \leq 3$; and confirming it for $n=4$ and $n=6$ in the case where one speed is much faster than the rest.  We also obtain a number of related results.
\end{abstract}

\maketitle

\section{Introduction}

The Lonely Runner Problem has been a popular research topic ever since it was introduced by Wills \cite{wills} and Cusick \cite{cusick}.  Its name comes from the following non-technical formulation.  Suppose $n$ runners start at the same point on a circular track of length $1$ and begin to run around the track at pairwise distinct constant speeds.  We deem a runner ``lonely'' at a certain time if their distance around the track from every other runner is at least $1/n$.  The Lonely Runner Conjecture asserts that regardless of the starting speeds, every runner gets lonely eventually (perhaps at different times for different runners).

Identify the circular track with $\mathbb{R}/\mathbb{Z}$, and consider the frame of reference of a single runner.  From this runner's perspective, it doesn't matter in which direction the other runners are going, so we may as well take all of their speeds to be positive.  Also, Bohman, Holzman, and Kleitman \cite{bohman} have shown that it suffices to consider only integer speeds.  Given a real number $x$, let $\Vert x \Vert$ denote the distance from $x$ to the nearest integer.  For a set of positive integer speeds $v_1, \ldots, v_n$, we define the associated \emph{maximum loneliness} to be
$$\ML(v_1, \ldots, v_n)=\max_{t \in \mathbb{R}}\min_{1 \leq i \leq n} \Vert tv_i \Vert.$$
(The maximum exists because $\min_{1 \leq i \leq n} \Vert tv_i \Vert$ is periodic in $t$.)  Then the Lonely Runner Conjecture can be expressed succinctly in terms of this quantity.
\begin{conjecture}[Lonely Runner Conjecture]\label{conj:old}
For any positive integers $v_1, \ldots, v_n$, we have $$\ML(v_1, \ldots, v_n) \geq \frac{1}{n+1}.$$
\end{conjecture}
If the Lonely Runner Conjecture is true, then the quantity $1/(n+1)$ is the best possible, for there are known equality cases (also called \emph{tight} sets of speeds) with $\ML(v_1, \ldots, v_n)=1/(n+1)$.  One such construction simply sets each $v_i=i$.

The Lonely Runner Conjecture is connected to questions in many fields.  We mention only a few: geometric view-obstruction (e.g., \cite{chen, cusick}); Diophantine approximation (e.g., \cite{betke, wills2, wills3, wills4}); flows in matroids (e.g., \cite{bienia, tarsi}); and chromatic numbers of distance graphs (e.g., \cite{liu, zhu}).  The conjecture has received substantial attention in recent decades.  Much of the research to date has one of the following flavors.
\begin{itemize}
\item Perhaps the most natural starting point is proving the Lonely Runner Conjecture for small values of $n$.  The $n=1$ case is trivial.  The $n=2$ and $n=3$ cases were proven by Betke and Wills \cite{betke}.  The $n=4$ case was proven first by Cusick and Pomerance \cite{pomerance} and later by Bienia, Goddyn, Gvozdjak, Seb\H{o}, and Tarsi \cite{bienia}.  A proof of the $n=5$ case is due to Bohman, Holzman, and Kleitman \cite{bohman}, and Renault \cite{renault} later found a simpler proof.  Most recently, Barajas and Serra \cite{barajas} proved the $n=6$ case.  The conjecture remains unresolved for $n \geq 7$.
\item Another appealing avenue is improving the trivial lower bound of $$\ML(v_1, \ldots, v_n) \geq \frac{1}{2n},$$ which holds because each runner is within $1/(2n)$ of the origin for only a $1/n$-fraction of all time.  Improvements are due to Chen \cite{chen}, Chen and Cusick \cite{chen2}, Perarnau and Serra \cite{perarnau}, and, most recently, Tao \cite{tao}, who showed that $1/(2n)$ can be replaced with
$$\frac{1}{2n}+\frac{c \log n}{n^2 \log \log n}$$
for some constant $c$ and all sufficiently large $n$.
\item A third approach is deriving conditions on the speeds $v_1, \ldots, v_n$ which guarantee $\ML(v_1, \ldots, v_n) \geq 1/(n+1)$.  Many such results impose ``lacunarity'' conditions on the speeds (i.e., require the speeds to grow at least as fast as a geometric progression).  See, e.g., Pandey \cite{pandey}, Ruzsa, Tuza, and Voigt \cite{ruzsa} Dubickas \cite{dubickas}, and Czerwi\'{n}ski \cite{czerwinski2}.  In a slightly different direction, Tao \cite{tao} proved several results on the case where all speeds are small.
\end{itemize}
Other investigations into the Lonely Runner Problem include: the work of Goddyn and Wong \cite{goddyn} on tight sets of speeds other than the trivial set $1, \ldots, n$; a result of Czerwi\'{n}ski \cite{czerwinski} on randomly chosen speeds; Tao's reduction \cite{tao} of the Lonely Runner Conjecture for $n$ runners to the case where all speeds are of size $\exp(O(n^2))$; Czerwi\'{n}ski and Grytczuk's proof \cite{grytczuk} of the Lonely Runner Conjecture if we are allowed to make one runner ``invisible'' at each time; and Chow and Rimani\'{c}'s resolution \cite{chow} of an analogous problem for function fields.
\\

In this paper, we introduce a new way to approach the Lonely Runner Problem.  A natural but (to our knowledge) hitherto unasked question is:
\begin{quote}
\textbf{Motivating Question.} If $v_1, \ldots, v_n$ satisfy the Lonely Runner Conjecture but do not form a set of tight speeds, must $\ML(v_1, \ldots, v_n)$ be uniformly bounded away from $1/(n+1)$?
\end{quote}
We conjecture that, contrary to what one might expect, this question has an affirmative answer.  (Speaking poetically, one might say that lonely runners are always either ``barely lonely'' or ``very lonely''.)  We in fact offer the following more precise statement about an unexpected rigidity of possible ``small'' values of $\ML(v_1, \ldots, v_n)$.
\begin{conjecture}[Loneliness Spectrum Conjecture]\label{conj:spectrum}
For any positive integers \linebreak $v_1, \ldots, v_n$, we have either $$\ML(v_1, \ldots, v_n)=\frac{s}{ns+1} \quad \text{for some } s \in \mathbb{N} \quad \text{or} \quad \ML(v_1, \ldots, v_n) \geq \frac{1}{n}.$$
\end{conjecture}
Note that this conjecture is strictly stronger than the Lonely Runner Conjecture.  In this paper, we will refer to Conjecture~\ref{conj:old} as the ``Lonely Runner Conjecture'', and we will refer to Conjecture~\ref{conj:spectrum} as the ``(Loneliness) Spectrum Conjecture".  We will often refer to the maximum loneliness amounts of the form $s/(ns+1)$ as the \emph{discrete part} of the (maximum) loneliness spectrum.

The Loneliness Spectrum Conjecture is well-motivated for a variety of reasons; most likely, it is new only because nobody posed the above question about near-equality cases in the past.
\begin{itemize}
\item It is natural to focus on maximum loneliness amounts in the interval \linebreak $[1/(n+1), 1/n)$ because if the Lonely Runner Conjecture is true, then this interval is precisely the ``new'' regime that is made available with the addition of the $n$-th runner.
\item It is fairly easy to compute all of the possible maximum loneliness amounts for $2$ moving runners, and these quantities form a discrete spectrum--exactly as predicted by the Spectrum Conjecture.  In the same vein, numerical experiments with small speeds support the Spectrum Conjecture.
\item The Spectrum Conjecture provides a ``good'' explanation for the appearance of the quantity $1/(n+1)$ in the Lonely Runner Conjecture: the maximum loneliness $1/(n+1)$ is the last element of a highly-structured discrete spectrum.  (On the old view, by contrast, it seems that the best motivation for conjecturing $1/(n+1)$ instead of a smaller quantity is simply that no set of speeds with smaller maximum loneliness is known.)  In other words, this new rigidity property points towards a new structural understanding of the Lonely Runner Problem.
\end{itemize}

\section{Main results}

We now give an overview of our main results.  In Section~\ref{sec:real}, we briefly discuss the situation in which one allows arbitrary positive real speeds.  In Section~\ref{sec:tools}, we collect some of the basic tools that we will use throughout the paper.  In Section~\ref{sec:spectrum}, we show that the Spectrum Conjecture is the best possible in the sense that for every $n$, the entire discrete part of the loneliness spectrum is attained.  In Sections~\ref{sec:2-runners} and~\ref{sec:3-runners}, we prove the Loneliness Spectrum Conjecture for $n=2$ and $n=3$, respectively.  (The $n=1$ case is trivial.)

\begin{theorem}\label{thm:2-and-3}
For any positive integers $v_1, v_2$, we have either $$\ML(v_1, v_2)=\frac{s}{2s+1} \quad \text{for some } s \in \mathbb{N} \quad \text{or} \quad\ML(v_1, v_2)= \frac{1}{2}.$$
For any positive integers $v_1, v_2, v_3$, we have either $$\ML(v_1, v_2, v_3)=\frac{s}{3s+1} \quad \text{for some } s \in \mathbb{N} \quad \text{or} \quad\ML(v_1, v_2, v_3)\geq \frac{1}{3}.$$
\end{theorem}
In Sections~\ref{sec:1-fast}-\ref{sec:reformulating}, we develop machinery for approaching the Spectrum Conjecture in the regime where one speed is much faster than the rest.  We eventually produce an explicit computation involving only the tight sets of $n-1$ speeds which allows us to decide whether or not the Spectrum Conjecture holds for $n$ runners in this regime.  Using this technique and previous results about tight speed sets for $3$ and $5$ runners, we establish the Spectrum Conjecture for $n=4$ and $n=6$ when one speed is much faster than the others.
\begin{theorem}\label{thm:4-5-6}
Let $n=4$ or $n=6$.  If $v_1< \cdots < v_n$ are positive integers satisfying $v_n>4v_{n-1}^4$, then we have either $$\ML(v_1, \ldots, v_n)=\frac{s}{ns+1} \quad \text{for some } s \in \mathbb{N} \quad \text{or} \quad \ML(v_1, \ldots, v_n) \geq \frac{1}{n}.$$
\end{theorem}
Along the way, we raise several questions about the set of all possible maximum loneliness amounts for $n$ runners, which we denote
$$\mathcal{S}(n)=\{\ML(v_1, \ldots, v_n): v_1, \ldots, v_n \text{ positive integers}\}.$$
In particular, we investigate where in $[0, 1/2]$ this set is dense.  One interesting consequence of this line of inquiry is that the Spectrum Conjecture for $n$ runners immediately implies the Lonely Runner Conjecture for $n-1$ runners as well as for $n$ runners.

In Sections~\ref{sec:short-progression} and~\ref{sec:interlude}, we recast previous results on the Lonely Runner Conjecture in terms of the Spectrum Conjecture.  We consider small speeds, randomly chosen speeds, and speeds that form lacunary sequences.  In Section~\ref{sec:conclusion}, we recapitulate some of the promising areas of future inquiry that arise in the intervening sections.

Often, the main barrier to proving a long-standing conjecture is that people try to establish the ``wrong'' version of the statement.  This tendency is especially true of induction-type arguments, where having a stronger induction hypothesis can make a proof easier.  We hope that the Loneliness Spectrum Conjecture may be the ``right'' way to approach the Lonely Runner Conjecture.

\section{Real speeds and integer speeds}\label{sec:real}

The original formulation of the Lonely Runner Problem allowed arbitrary positive real speeds, but it was quickly realized that the case of integer speeds tells the whole story.  We can extend our definition of maximum loneliness to general real speeds as follows: given a set of positive real speeds $v_1, \ldots, v_n$, define the associated maximum loneliness to be
$$\ML(v_1, \ldots, v_n)=\sup_{t \in \mathbb{R}}\min_{1 \leq i \leq n} \Vert tv_i \Vert.$$
We remark that the runners-going-around-a-track formulation of the Lonely Runner Conjecture is slightly stronger than the assertion that
$$\ML(v_1, \ldots, v_n) \geq \frac{1}{n+1}$$
for all choices of positive real $v_1, \ldots, v_n$, since the supremum may not be attained.  This consideration is not troubling, however, due to the following result of Bohman, Holzman, and Kleitman \cite{bohman} from 2001; it is unclear whether or not such a result had been established prior.

\begin{lemma}[Bohman, Holzman, and Kleitman \cite{bohman}]\label{lem:real-speeds}
Fix some integer $n \geq 2$, and let $\delta$ be a positive real number such that $\ML(v_1, \ldots, v_{n-1})> \delta$ for all positive integers $v_1, \ldots, v_{n-1}$.  If $v'_1, \ldots, v'_n$ are positive real numbers and some quotient $v'_i/v'_j$ is irrational, then $\ML(v'_1, \ldots, v'_n)>\delta$.
\end{lemma}
In particular, such an ``irrational'' set of speeds cannot have the smallest maximum loneliness, among all sets of $n$ speeds.  It immediately follows from (say) the $\delta=2/(2n+1)$ case of this lemma that the Lonely Runner Conjecture for integer speeds (as stated in the Introduction) implies the the runners-going-around-a-track formulation for real speeds.

\begin{corollary}[Bohman, Holzman, and Kleitman \cite{bohman}]\label{cor:real-speeds}
Fix some integer $n \geq 2$, and assume that the Lonely Runner Conjecture (for integer speeds) holds for both $n$ and $n-1$ runners.  Then for any positive real speeds $v_1, \ldots, v_n$, there is a time $t$ such that $\Abs{tv_i} \geq 1/(n+1)$ for all $1 \leq i \leq n$.
\end{corollary}

We can also use the above extended definition of maximum loneliness to formulate a version of the Loneliness Spectrum Conjecture for real speeds.  As in the case of the Lonely Runner Conjecture, it turns out that checking integer speeds is sufficient.  We need the following proposition, which in Section~\ref{sec:accumulation} we will deduce as an easy corollary of our results on accumulation points.

\begin{prop}\label{prop:real-speeds}
Fix some integer $n \geq 2$, and assume that the Spectrum Conjecture (for integer speeds) holds for $n$ runners.  Then the Lonely Runner Conjecture (for integer speeds) holds for $n-1$ runners.
\end{prop}

We can reduce the ``real'' case of the Spectrum Conjecture to the ``integer'' case.

\begin{theorem}\label{thm:real-speeds}
Fix some integer $n \geq 2$, and assume that the Spectrum Conjecture (for integer speeds) holds for $n$ runners.  Then for any positive real speeds $v_1, \ldots, v_n$, we have either $$\ML(v_1, \ldots, v_n)=\frac{s}{ns+1} \quad \text{for some } s \in \mathbb{N} \quad \text{or} \quad \ML(v_1, \ldots, v_n) \geq \frac{1}{n}.$$
\end{theorem}

\begin{proof}
First, suppose every quotient $v_i/v_j$ is rational.  Then let $\ell$ be the lcm of the denominators appearing in $v_2/v_1,\ldots, v_n/v_1$, and note that
$$\ML(v_1, \ldots, v_n)=\ML\left(\ell, \frac{\ell v_2}{v_1}, \ldots, \frac{\ell v_n}{v_1} \right),$$
where the speeds on the right-hand side are all integers.  So, by assumption, the maximum loneliness is as in the statement of the theorem.

Second, suppose not every $v_i/v_j$ is rational.  By Proposition~\ref{prop:real-speeds}, we know that the Lonely Runner Conjecture holds for $n-1$ runners.  Now, applying Lemma~\ref{lem:real-speeds} with a sequence of $\delta$'s that approaches $1/n$ from below shows that $\ML(v_1, \ldots, v_n) \geq 1/n$, as desired.
\end{proof}

In light of this result, we concern ourselves with the case of integer speeds in the remainder of this paper.

\section{Tools and Preliminary Observations}\label{sec:tools}

We begin with a few simple observations that we will use freely later.  We always take $v_1, \ldots, v_n$ to be positive integers, with $n \geq 2$.

First, note that without loss of generality we can restrict our attention to the case where the speeds $v_1, \ldots, v_n$ do not all share a common factor; indeed, if $\gcd(v_1, \ldots, v_n)=g>1$, then $\ML(v_1, \ldots, v_n)=\ML(v_1/g, \ldots, v_n/g)$, where each $v_i/g$ is an integer.  We also lose nothing by taking the speeds to be pairwise distinct.

Second, recall that for fixed $v_1, \ldots, v_n$, we are looking for the real number (time) $t$ that maximizes the function
$$f(t)=\min_{1 \leq i \leq n}\Vert tv_i \Vert.$$
Since $f(t+1)=f(t)$, it suffices to consider $t$ in the interval $[0,1)$.  In fact, since $f(t)=f(-t)$, we could further restrict our attention to $[0, 1/2]$.  Moreover, all local maxima of $f$ should occur at times where there are two runners of minimum distance to the origin and these runners are on ``different'' halves of the circular track, for otherwise we could obtain a larger loneliness by perturbing the time.  We make this observation precise in the following proposition.
\begin{prop}\label{prop:times-to-check}
Let $v_1, \ldots, v_n$ be positive integers ($n \geq 2$) with $\gcd(v_1, \ldots, v_n)=1$.  Then every local maximum of the function
$$f(t)=\min_{1 \leq i \leq n}\Vert tv_i \Vert$$
occurs at a time of the form
$$t_0=\frac{m}{v_i+v_j},$$
where $1 \leq i<j \leq n$ and $m$ is an integer.
\end{prop}

\begin{proof}
Fix a time $t_0 \in \mathbb{R}$ at which $f(t)$ achieves a local maximum.  We first consider the case $f(t_0)=1/2$.  Write $t_0=a/b$, where $\gcd(a,b)=1$ and $b>0$.  Then $av_i \equiv b/2 \pmod{b}$ for all $1 \leq i \leq n$.  In particular, $b$ is even, whence $a$ is odd.  There is an integer $c$ such that $ac \equiv 1 \pmod{b}$.  Since $\gcd(v_1, \ldots, v_n)=1$, there exist integers $k_1, \ldots, k_n$ such that $k_1v_1+\cdots+k_nv_n=c$.  Then
$$\frac{b}{2} (k_1+\cdots+k_n)\equiv a(k_1v_1+\cdots+k_nv_n) \equiv ac \equiv 1 \pmod{b},$$
which implies that $b/2$ is relatively prime to $b$.  So in fact $b=2$, and we conclude that all of the $v_i$'s are odd.  (Incidentally, this argument shows that $\ML(v_1, \ldots, v_n)=1/2$ if and only if every $v_i$ is odd.)  We know that (for instance) $v_1+v_2$ is even, so we can express $t_0=a/2$ as $t_0=m/(v_1+v_2)$ for some integer $m$.

We now consider the case where $f(t_0)<1/2$ strictly.  The equality $\Abs{t_0v_i}=f(t_0)$ can be realized only as $t_0v_i=\ell+f(t_0)$ or $t_0v_i=\ell-f(t_0)$, where $\ell$ is an integer.  Suppose that as $i$ ranges over $1, \ldots, n$, only the first type of equality is realized.  Then for some $\varepsilon>0$, the fractional part of every $t_0v_i$ is contained in the interval $[f(t_0), 1-f(t_0)-\varepsilon]$.  For $\eta=\varepsilon/(2 \max_i v_i)$, we obtain the contradiction $f(t_0+\eta)>f(t_0)$.  Similarly, if only the second type of equality is realized, then we obtain the contradiction $f(t_0-\eta)>f(t_0)$.  So we conclude that both types of equality are realized, and there exist indices $i,j$ and integers $k, \ell$ such that $t_0v_i=k+f(t_0)$ and $t_0v_j=\ell-f(t_0)$.  (Note that $i \neq j$ since $\Abs{t_0v_i} \neq \Abs{t_0v_j}$; here, we used $f(t_0)< 1/2$.)  Summing and solving for $t_0$ gives
$$t_0=\frac{k+\ell}{v_i+v_j},$$
as desired.
\end{proof}

Of course, the global maximum is among these local maxima for $t$ in the interval $[0,1)$.  This proposition shows that we can determine $\ML(v_1, \ldots, v_n)$ by checking only finitely many times, where this number grows at most cubically with the size of the fastest speed.  This observation is particularly useful for performing computational experiments.  Another advantage of this perspective is that it gives us a useful way to look at candidate times in ``chunks'' (according to the pairs $i,j$) instead of all at once.  This way of thinking often reveals underlying structure that is otherwise opaque.

Third, we mention the concept of a \emph{pre-jump}, as used by Bienia, Goddyn, Gvozdjak, Seb\H{o}, and Tarsi \cite{bienia}.  Their insight is essentially that if we know the values of $\Abs{t_1v_i}$ and $\Abs{t_2v_i}$, then we can say something about $\Abs{(t_1+\alpha t_2)v_i}$, where $\alpha$ is an integer.  This is particularly useful with many of the quantities involved are zero.  For instance, if the speeds $v_1, \ldots, v_r$ are all divisible by $g$ and the speeds $v_{r+1}, \ldots, v_n$ are not, then adding multiples of $1/g$ to a time ``moves'' the runners with speeds $v_{r+1}, \ldots, v_n$ while ``fixing'' those with speeds $v_1, \ldots, v_r$.

\section{Achieving the spectrum}\label{sec:spectrum}
A bit of experimentation suggests an explicit construction that achieves the entire discrete part of the spectrum in the Loneliness Spectrum Conjecture.  The motivating idea is that one can obtain small maximum loneliness values with $n$ runners by starting with a tight speed set for $n-1$ runners and choosing the last speed $v_n$ so that $\Abs{t_0v_n}=0$ at every ``equality time'' $t_0$ for the first $n-1$ runners.

\begin{theorem}\label{thm:spectrum-construction}
For every integer $n \geq 2$ and every natural number $s$, we have $$\ML(1,2,\ldots, n-1, ns)=\frac{s}{ns+1}.$$
\end{theorem}

\begin{proof}
The $s=1$ case follows from the known tight case $\ML(1,\ldots, n)=1/(n+1)$, so we restrict our attention to $s \geq 2$.  The proof is a straightforward computation with Proposition~\ref{prop:times-to-check}.  We have to check times with denominators between $1$ and $2n-3$ and between $ns+1$ and $ns+n-1$.  As before, let $f(t)$ denote the loneliness at time $t$.
\begin{itemize}
\item Consider $t=m/d$, for $1 \leq d \leq 2n-3$.  We know that $f(t) \leq 1/n$ because of the speeds $1, \ldots, n-1$.  Since $f(t)$ is a nonnegative rational number with denominator at most $2n-3$, we must have either $f(t)=1/n$ or $f(t)\leq 1/(n+1)$.  We could have $f(t)=1/n$ only for $t=m/n$, but in this case $\Abs{t(ns)}=0$ gives $f(t)=0$, so we conclude that $f(t) \leq 1/(n+1)$.
\item Consider $t=m/(ns+j)$, for $1 \leq j \leq n-1$.  Recall that $f(t) \leq 1/n$.  So the maximum possible value for $f(t)$ is $s/(ns+j)$, and in fact this value is achieved for $m=s$.
\end{itemize}
Taking the maximum value of $f(t)$ among these possibilities gives that $$\ML(1,2, \ldots, n-1, ns)=\frac{s}{ns+1},$$ as desired.
\end{proof}
We remark that this spectrum can be achieved in much the same way by starting with other known equality cases and adding a new fast runner (see Goddyn and Wong \cite{goddyn}); the case work in the computation becomes more extensive, and sometimes a few of the values for small $s$ are not obtained.  We return to this idea in Section~\ref{sec:1-fast}.

\section{Two moving runners}\label{sec:2-runners}
The $n=2$ case of the Spectrum Conjecture is mentioned in a passing remark of Bienia, Goddyn, Gvozdjak, Seb\H{o}, and Tarsi \cite{bienia}; for the sake of completeness, we provide a proof in the language introduced above.

\begin{theorem}\label{thm:2-runners}
Let $v_1, v_2$ be relatively prime positive integers.  If $v_1$ and $v_2$ are both odd, then $$\ML(v_1, v_2)=\frac{1}{2}.$$
Otherwise, $$\ML(v_1, v_2)=\frac{s}{2s+1},$$ where $2s+1=v_1+v_2$.
\end{theorem}

\begin{proof}
If $v_1$ and $v_2$ are both odd, then $\Abs{(1/2) v_1}=\Abs{(1/2) v_2}=1/2$ gives $\ML(v_1,v_2)=1/2$, so we restrict our attention to the case where $v_1$ and $v_2$ are not both odd.  In particular, their sum is odd.  By Proposition~\ref{prop:times-to-check}, we have to check only times $t=m/(v_1+v_2)$, where we know that
$$\Abs{tv_1}=\Abs{tv_2}=\Abs{\frac{mv_1}{v_1+v_2}}.$$
Note that this quantity is always at most $s/(2s+1)$, where $2s+1=v_1+v_2$.  Since $v_1$ and $v_2$ are relatively prime, we also have that $v_1$ is relatively prime to $v_1+v_2$.  So there exists an integer $m$ such that $mv_1 \equiv s \pmod{2s+1}$, which shows that the loneliness $s/(2s+1)$ is attained.
\end{proof}

\section{Three moving runners}\label{sec:3-runners}
The $n=3$ case of the Spectrum Conjecture is much more delicate than the $n=2$ case.  The main idea is that we use a pre-jump to handle the case where two speeds share a large common factor, after which we can control the remaining cases more precisely.  We require the following technical lemma, which says that if we consider all times with denominator the sum of two fixed speeds, then we should get a loneliness of at least $1/3$, up to a rounding error, unless the third runner always ``stays'' at $0$.

\begin{lemma}\label{lem:3-runners}
Let $v_1, v_2, v_3$ be positive integers with $\gcd(v_1,v_2,v_3)=1$, and suppose no two of these speeds have a common factor greater than $2$.  Let $$r=\left\lfloor \frac{v_1+v_2}{3} \right\rfloor,$$ and let $L$ denote the maximum loneliness that is achieved at a time of the form $$t=\frac{m}{v_1+v_2}$$ ($m \in \mathbb{Z}$).  Then we have the following dichotomy:
\begin{itemize}
\item If $v_3$ is a multiple of $v_1+v_2$, then $L=0$.
\item If $v_3$ is not a multiple of $v_1+v_2$, then $L \geq r/(v_1+v_2)$.
\end{itemize}
\end{lemma}

\begin{proof}
The first statement is trivial.  For second statement, fix $v_1, v_2, v_3$ with $v_3$ not a multiple of $v_1+v_2$.  We condition on $\gcd(v_1, v_2)$, which must be $1$ or $2$ by assumption.  Recall that at every $t=m/(v_1+v_2)$, we have $\Abs{tv_1}=\Abs{tv_2}$.

First, suppose $\gcd(v_1,v_2)=1$.  Then there is an integer $u$ such that $uv_1 \equiv 1 \pmod{v_1+v_2}$.  Write $m \equiv \ell u \pmod{v_1+v_2}$, so that $\ell$ ranges over the residues modulo $v_1+v_2$ as $m$ does so, and consider the times $$t=\frac{\ell u}{v_1+v_2}.$$
Then for all $r \leq \ell \leq v_1+v_2-r$, we have that $\Abs{tv_1}$ (equivalently, $\Abs{tv_2}$) is at least $r/(v_1+v_2)$.  We now claim that $$\Abs{tv_3}\geq \frac{r}{v_1+v_2}$$ for some $\ell$ in this range.  In other words, some element of $$ruv_3, (r+1)uv_3, \ldots, (v_1+v_2-r)uv_3$$ leaves a residue between $r$ and $v_1+v_2-r$ modulo $v_1+v_2$.

How this comes about depends on the residue of $uv_3$ modulo $v_1+v_2$.  We may take this residue to be between $1$ and $(v_1+v_2)/2$ since otherwise we can replace $v_3$ with $c(v_1+v_2)-v_3$ for some large integer $c$ that makes this quantity positive. We handle the various possibilities separately:
\begin{itemize}
\item Suppose $uv_3 \equiv 1 \pmod{v_1+v_2}$.  Then $ruv_3 \equiv r \pmod{v_1+v_2}$, as desired.
\item Suppose $uv_3 \equiv j \pmod{v_1+v_2}$, for some $2 \leq j \leq v_1+v_2-2r+1$.  The upper bound on $j$ tells us that if $\ell j<k(v_1+v_2)+r$ for some integer $k$, then $(\ell+1)j \leq (k+1)(v_1+v_2)-r$.  In other words, incrementing $\ell$ cannot make $\ell j$ ``skip'' over all of the residue classes (strictly) between $r$ and $v_1+v_2-r$.  So it remains only to show that the $\ell j$'s are not all contained in an interval of the form $$k(v_1+v_2)-r+1, \ldots, k(v_1+v_2)+r-1$$ for any integer $k$.  The difference between the largest and smallest elements of this interval is $2r-2$.  At the same time, the lower bound on $j$ gives
$$(v_1+v_2-r)j-rj \geq 2(v_1+v_2-2r) \geq 2r>2r-2,$$
so the $\ell j$'s cannot all be contained in such a short interval.  We conclude that some $\ell j$ leaves a residue between $r$ and $v_1+v_2-r$, as desired.
\item Suppose $uv_3 \equiv j \pmod{v_1+v_2}$, for some $v_1+v_2-2r+2 \leq j <(v_1+v_2)/2$.  Recall that for $r \leq \ell \leq v_1+v_2-r$, we are done if the quantity $\ell j$ ever leaves a residue between $r$ and $v_1+v_2-r$.  In particular, this possibility obtains if there is any $r \leq \ell \leq v_1+v_2-r-1$ such that the residue of $\ell j$ is between $r-j$ and $v_1+v_2-r-j$, for then the residue of $(\ell+1)j$ is between $r$ and $v_1+v_2-r$.  So it suffices to show that the residues of $\ell j$, for $r \leq \ell \leq v_1+v_2-r-1$, cannot be confined to the intervals
$$I_1=\{v_1+v_2-r-j+1, v_1+v_2-r-j+2, \ldots, r-1\}$$
and
$$I_2=\{v_1+v_2-r+1, v_1+v_2-r+1, \ldots, v_1+v_2+r-j-1\}.$$
Note that the difference between the largest and smallest elements of $I_1$ is
$$(r-1)-(v_1+v_2-r-j+1)<\frac{1}{6}(v_1+v_2)-2<j,$$
where we used the upper bound on $j$.  Similarly, the difference between the largest and smallest elements of $I_2$ is
$$(r-j-1)-(-r+1)=2r-j-2<j.$$
These bounds imply that consecutive residues of $\ell j$ and $(\ell+1)j$ cannot both lie in a single one of these two intervals, so we have to worry about only the possibility in which $\ell j$ alternately lies in $I_1$ and $I_2$ as $\ell$ grows from $r$ to $v_1+v_2-r-1$.  If this were the case, we would have at least
$$\left\lfloor \frac{(v_1+v_2-r-1)-r}{2} \right\rfloor \geq \frac{1}{6}(v_1+v_2)-1$$
values of $\ell$ (increasing in increments of $2$) with $\ell j$ leaving a residue in $I_1$.  Note that as $\ell$ increases by $2$, the residue of $\ell j$ increases by
$$h=v_1+v_2-2j,$$
which is nonzero by the condition on $j$.  So the difference between the the largest and smallest of these residues is at least $\frac{1}{6}(v_1+v_2)-2$, but then it is impossible to fit this entire arithmetic progression into $I_1$.  So we conclude that in fact some $\ell j$ leaves a residue between $r$ and $v_1+v_2-r$, as desired.
\item Suppose $uv_3 \equiv (v_1+v_2)/2 \pmod{v_1+v_2}$.  Then either $ruv_3$ or $(r+1)uv_3$ leaves a residue of $(v_1+v_2)/2$ modulo $v_1+v_2$, and this is certainly between $r$ and $v_1+v_2-r$.
\end{itemize}
This concludes the argument for the $\gcd(v_1,v_2)=1$ case.
\\

Second, suppose $\gcd(v_1,v_2)=2$.  Then there is an integer $u$ such that $uv_1 \equiv 2 \pmod{v_1+v_2}$.  Note that $v_3$ is odd due to our gcd restrictions.  Now, write even $m$ as $m \equiv \ell u \pmod{v_1+v_2}$, so that $\ell$ ranges over the residues $1,2,\ldots, (v_1+v_2)/2$ modulo $v_1+v_2$ as $m$ ranges over the even residues modulo $v_1+v_2$, and consider times $t=(\ell u)/(v_1+v_2)$.  Then for all $r/2 \leq \ell \leq (v_1+v_2-r)/2$, we have that $\Abs{tv_1}$ (equivalently, $\Abs{tv_2}$) is at least $r/(v_1+v_2)$.  We now claim that either $$\Abs{tv_3}\geq \frac{r}{v_1+v_2} \quad \text{or} \quad \Abs{tv_3} \leq \frac{1}{2}-\frac{r}{v_1+v_2}$$
for some $\ell$ in this range.  The second possibility is sufficient to establish the desired result because at the time $t+ 1/2$ (which is still of the form $m/(v_1+v_2)$), we have
$$\Abs{\left(t+\frac{1}{2}\right)v_1}=\Abs{tv_1}, ~ \Abs{\left(t+\frac{1}{2}\right)v_2}=\Abs{tv_2}, ~ \text{and} ~\Abs{\left(t+\frac{1}{2}\right)v_3}=\frac{1}{2}-\Abs{t_1v_3}.$$
(We can think of this manipulation as a pre-jump with the times $t$ and $1/2$.)  So our claim is that some element of
$$\left(\frac{r}{2} \right)uv_3, \left(\frac{r}{2}+1\right)uv_3, \ldots, \left(\frac{v_1+v_2-r}{2}\right)uv_3$$
leaves a residue between $r$ and $v_1+v_2-r$ or between $r-(v_1+v_2)/2$ and $(v_1+v_2)/2-r$ modulo $v_1+v_2$.

As before, we divide cases according to the residue of $uv_3$, where we can take this residue to be between $1$ and $(v_1+v_2)/2$.  Because the arguments are essentially the same as what we presented above in the $\gcd(v_1,v_2)=1$ case, we provide only sketches.
\begin{itemize}
\item Suppose $uv_3 \equiv j \pmod{v_1+v_2}$, for some $1 \leq j \leq v_1+v_2-2r+1$.  The upper bound on $j$ tells us that incrementing $\ell$ cannot make $\ell j$ ``skip'' over either of the two forbidden intervals of residues, so we have to worry about only the cases where the $\ell j$'s are either all contained in
$$\frac{v_1+v_2}{2}-r+1, \frac{v_1+v_2}{2}-r+2, \ldots, r-1$$
or all contained in
$$v_1+v_2-r+1, v_1+v_2-r+2, \ldots, \frac{v_1+v_2}{2}+r-1.$$
But neither of these intervals is long enough to contain the entire arithmetic progression of $\ell j$'s.
\item Suppose $uv_3 \equiv j \pmod{v_1+v_2}$, for some $v_1+v_2-2r+2 \leq j <(v_1+v_2)/2$.  The argument then goes roughly as in the third bullet above, except that we now have $h \geq 2$ since $v_1+v_2$ is even.
\item Suppose $uv_3 \equiv (v_1+v_2)/2 \pmod{v_1+v_2}$.  Then $ruv_3$ leaves a residue of $0$ or $(v_1+v_2)/2$ modulo $v_1+v_2$, either of which is sufficient.
\end{itemize}
This concludes the argument for the $\gcd(v_1,v_2)=2$ case.
\end{proof}

Lemma~\ref{lem:3-runners} will handle most of the ``difficult'' sets of speeds in the $n=3$ case of the Spectrum Conjecture.

\begin{theorem}\label{thm:3-runners}
Let $v_1, v_2, v_3$ be positive integers with $\gcd(v_1, v_2, v_3)=1$.  Then we have either $$\ML(v_1, v_2, v_3)=\frac{s}{3s+1} \quad \text{for some } s \in \mathbb{N} \quad \text{or} \quad \ML(v_1, v_2, v_3) \geq \frac{1}{3}.$$
\end{theorem}

\begin{proof}
First of all, suppose some two of the speeds have a common factor of at least $3$, say, $\gcd(v_1,v_2)=g \geq 3$.  Note that $\gcd(g,v_3)=1$.  By Theorem~\ref{thm:2-runners}, there exists a time $t$ such that both $\Abs{tv_1}$ and $\Abs{tv_2}$ are at least $1/3$.  By the Pigeonhole Principle, there is an integer $h$ such that $$\Abs{\left(t+\frac{h}{g}\right)v_3} \geq \frac{1}{2}-\frac{1}{2g} \geq \frac{1}{3}.$$
We also know that $\Abs{(t+h/g)v_1}=\Abs{tv_1}$ and $\Abs{(t+h/g)v_2}=\Abs{tv_2}$, so the loneliness at time $t+h/g$ is at least $1/3$. (We are using a pre-jump with the times $t$ and $1/g$.)  We conclude that $\ML(v_1, v_2, v_3) \geq 1/3$.  Henceforth, we restrict our attention to the case where no two speeds have a common factor greater than $2$.

Next, suppose no speed is a multiple of $3$.  Then the loneliness at time $t=1/3$ is at least $1/3$, and we are done.  So we can restrict our attention to the case where exactly one speed is a multiple of $3$.

For each $1 \leq i<j \leq 3$, let $$r_{i,j}=\left\lfloor \frac{v_i+v_j}{3} \right\rfloor,$$ and let $L_{i,j}$ denote the maximum loneliness that is achieved at a time of the form $t=m/(v_i+v_j)$.  Let $k$ be the remaining element of $\{1,2,3\}$.  Lemma~\ref{lem:3-runners} provides the following dichotomy for each pair $i,j$:
\begin{itemize}
\item If $v_k$ is a multiple of $v_i+v_j$, then $L_{i,j}=0$.
\item If $v_k$ is not a multiple of $v_i+v_j$, then $L_{i,j} \geq r_{i,j}/(v_i+v_j)$.
\end{itemize}
Recall that $\ML(v_1, v_2, v_3)$ is the maximum of the three values $L_{i,j}$.  If any $L_{i,j} \geq 1/3$, then we are done, so we restrict our attention to the case where this does not occur.   In particular, the second case of the dichotomy collapses to $$L_{i,j}=\frac{r_{i,j}}{v_i+v_j}$$ if $v_i+v_j$ is not a multiple of $3$, and the second case becomes completely disallowed if $v_i+v_j$ is a multiple of $3$.

Suppose the second possibility of the dichotomy obtains for each pair $i,j$, i.e., each $L_{i,j}=r_{i,j}/(v_i+v_j)$, where $v_i+v_j$ is not a multiple of $3$.  Thus, the residues of $v_1, v_2, v_3$ modulo $3$ are either $1,1,0$ or $2,2,0$.  In the first scenario, write $v_1=3a+1$, $v_2=3b+1$, and $v_3=3c$, where $a<b$.  Then we have
$$L_{1,2}=\frac{a+b}{3a+3b+2}, \quad L_{1,3}=\frac{a+c}{3a+3c+1}, \quad \text{and} \quad L_{2,3}=\frac{b+c}{3b+3c+1}.$$
Direct comparison shows that $L_{2,3}$ is the largest of these three quantities, so $$\ML(v_1,v_2,v_3)=\frac{b+c}{3b+3c+1},$$ as desired.

In the second scenario, write $v_1=3a+2$, $v_2=3b+2$, and $v_3=3c$, where $a<b$.  Then we have
$$L_{1,2}=\frac{a+b+1}{3a+3b+4}, \quad L_{1,3}=\frac{a+c}{3a+3c+2}, \quad \text{and} \quad L_{2,3}=\frac{b+c}{3b+3c+2}.$$
We are done if $L_{1,2}$ is the largest of these three quantities.  Assume (for contradiction) this does not occur; since $L_{2,3}>L_{1,3}$, we must have $L_{2,3}>L_{1,2}$.  Direct computation gives the inequality
$$c>2a+b+2.$$
Now, consider the time $t=1/3-1/(9c)$.  We compute that the fractional parts of $tv_1$ and $tv_2$ are, respectively,
$$\frac{2}{3}-\frac{3a+2}{9c} \quad \text{and} \quad \frac{2}{3}-\frac{3b+2}{9c},$$
whence both $\Abs{tv_1}$ and $\Abs{tv_2}$ are greater than $1/3$.  Also,
$$\Abs{tv_3}=0+\frac{3c}{9c}=\frac{1}{3}.$$
This alltogether implies that $\ML(v_1,v_2,v_3) \geq 1/3$, contrary to our assumption.  So $L_{1,2}$ must be the largest, as desired.  This exhausts the cases in which the second possibility of the dichotomy obtains for each pair $i,j$.

It remains to treat the case in which the first possibility of the dichotomy obtains for some pair, say, $i=1$, $j=2$.  Then we have
$$L_{1,2}=0, \quad L_{1,3}=\frac{r_{1,3}}{v_1+v_3}, \quad \text{and} \quad L_{2,3}=\frac{r_{2,3}}{v_2+v_3}.$$
If both $v_1+v_3$ and $v_2+v_3$ are equivalent to $1$ modulo $3$, then we are done.  Similarly, if $v_1+v_3 \equiv 1 \pmod{3}$ and $v_2+v_3 \equiv 2 \pmod{3}$, then $L_{1,3}>L_{2,3}$ by direct computation, and we are also done.

So it remains only to treat the case where both $v_1+v_3$ and $v_2+v_3$ are equivalent to $2$ modulo $3$.  In particular, $v_1$ and $v_2$ leave the same residue modulo $3$, so this residue is not $0$.  This in turn implies that $v_3$ is a multiple of $3$ (since exactly one of the speeds is a multiple of $3$).  Then $v_1$ and $v_2$ leave a remainder of $2$ modulo $3$.  At the time $t=1/3-1/(3v_3)$, we obtain a loneliness of $1/3$ (as above), which contradicts our earlier assumption.  This completes the proof.
\end{proof}

A close inspection of the previous two proofs reveals that a maximum loneliness of $1/4$ is obtained only for the speeds $1,2,3$.
\begin{corollary}\label{cor:3-runners}
The only tight set of speeds for $n=3$ (up to scaling) is $1,2,3$.
\end{corollary}
\begin{proof}[Proof (sketch).]
If the largest $L_{i,j}$ is $L_{1,3}=1/4$, then we must have $v_1+v_3=4$, which implies that (without loss of generality) $v_1=1$ and $v_3=3$.  Direct computation shows that $\ML(1,v_2,3)>1/4$ whenever $v_2 \geq 4$: consider times $5/12<t<7/12$, where both $\Abs{t}$ and $\Abs{3t}$ are greater than $1/4$; if in this interval the $v_2$ runner traverses a distance greater than $1/2$, then we find a time with loneliness larger than $1/4$.
\end{proof}

In principle, one could reproduce the program of this section for $4$ or more moving runners: if many of the speeds share a large common factor, then the induction hypothesis together with a pre-jump gives the desired result; otherwise, there are only finitely many cases to consider in establishing an analog of Lemma~\ref{lem:3-runners}, after which ad hoc arguments could take care of the remaining sporadic cases.  Given the difficulty and length of the proof for $n=3$, however, this program is probably infeasible for $n \geq 4$, as least for a non-computer-assisted proof.

\section{One very fast runner}\label{sec:1-fast}

It is natural to try to use induction for the Lonely Runner Problem.  One appealing strategy is the following: given speeds $v_1, \ldots, v_n$, use an induction hypothesis to obtain a lower bound for $\ML(v_1, \ldots, v_{n-1})$, with this loneliness achieved at some time $t_0$, then modify $t_0$ in order to obtain a time $t_1$ where every $\Abs{t_1v_i}$ (now for $1 \leq i \leq n)$ remains large.  The following innocuous proposition demonstrates how this approach could play out if one runner is much faster than the rest.
\begin{prop}\label{prop:perturb}
Let $v_1<\cdots < v_{n-1}$ be positive integers ($n \geq 2$) with \linebreak $\ML(v_1, \ldots, v_{n-1}) \geq L$, and fix some $0< \varepsilon<L$.  Then we have that
$$\ML(v_1, \ldots, v_n)\geq L-\varepsilon$$
whenever
$$v_n \geq \left(\frac{L-\varepsilon}{\varepsilon}\right)v_{n-1}.$$
\end{prop}

\begin{proof}
Choose a time $t_0$ such that $\Abs{t_0v_i}\geq L$ for all $1 \leq i \leq n-1$.  We know that in a time interval of length $\varepsilon/v_{n-1}$, each such runner traverses a distance of at most $\varepsilon$.  Let
$$I=\left[t_0-\frac{\varepsilon}{v_{n-1}}, t_0+\frac{\varepsilon}{v_{n-1}}\right]$$
be the closed interval of all times at most $\varepsilon/v_{n-1}$ away from $t_0$.  Consequently, for every $t \in I$, we have $\Abs{tv_i} \geq L-\varepsilon$ for $1 \leq i \leq n-1$.  In any interval of length $(2\varepsilon)/v_{n-1}$, the runner with speed $v_n$ traverses a distance of
$$\frac{2v_n\varepsilon}{v_{n-1}} \geq 2(L-\varepsilon),$$
which in particular implies that there is some $t \in I$ with $\Abs{tv_n} \geq L-\varepsilon$.
\end{proof}

In order to demonstrate the use of this proposition, we mention an immediate consequence.
\begin{prop}\label{prop:factorial-growth}
Fix real numbers $\alpha>0$ and $c>-1$ such that $$\frac{\alpha}{1+c} \leq \frac{1}{2}.$$
If $v_1, \ldots, v_n$ are positive integers satisfying
$$\frac{v_i}{v_{i-1}} \geq i+c-1$$
for all $2 \leq i \leq n$, then $\ML(v_1, \ldots, v_n)\geq \alpha/(n+c)$.
\end{prop}

\begin{proof}
Use induction on $n$ and apply Proposition~\ref{prop:perturb}.
\end{proof}
The case $\alpha=1$, $c=1$ tells us that the Lonely Runner Conjecture is satisfied for any set of speeds that grows factorially.  The same argument shows that $\ML(v_1, \ldots, v_n) \geq 1/n$ if $v_1$ and $v_2$ have the same $2$-adic valuation and the subsequent $v_i$'s grow factorially.  We remark that these two propositions also hold for general real speeds.

This discussion motivates giving special attention to the case where one speed is significantly larger than the rest.  More precisely, we examine the following weak version of the Loneliness Spectrum Conjecture.
\begin{conjecture}\label{conj:1-fast}
For every integer $n \geq 4$, there exists a function $f_n: \mathbb{N} \to \mathbb{N}$ such that the following holds: for any positive integers $v_1<\cdots <v_n$ with $v_n>f_n(v_{n-1})$, we have either $$\ML(v_1, \ldots, v_n)=s/(ns+1) \quad \text{for some } s \in \mathbb{N} \quad \text{or} \quad \ML(v_1, \ldots, v_n) \geq 1/n.$$
\end{conjecture}
(Note that we do not require the speeds to lack a common factor.)  In this section and the next two sections, we develop an explicit way to determine whether or not this conjecture holds for $n$ runners based on the tight speed sets for $n-1$ runners.

The conjecture immediately splits into two cases, depending on whether or not $\ML(v_1, \ldots, v_{n-1})$ is strictly larger than $1/n$.  For the case $\ML(v_1, \ldots, v_{n-1})>1/n$, we quickly obtain an affirmative answer to Conjecture~\ref{conj:1-fast} with $f_n(v)$ quadratic in $v$ (and independent of $n$).
\begin{lemma}\label{lem:1-fast}
Let $v_1<\cdots < v_{n-1}$ be positive integers ($n \geq 3$) with $$\ML(v_1, \ldots, v_{n-1})=L>\frac{1}{n}.$$
Then we have
$$ML(v_1, \ldots, v_n) \geq \frac{1}{n}$$
whenever 
$$v_n \geq v_{n-1}(2v_{n-1}-1).$$
\end{lemma}

\begin{proof}
By Proposition~\ref{prop:times-to-check}, we know that $$L=\frac{m}{v_i+v_j}$$ for some integer $m$ and some $1 \leq i<j \leq n-1$.  In particular, the inequality $$nm>v_i+v_j$$
in the integers implies that 
$$nm \geq v_i+v_j+1.$$
We then compute
$$L \geq \frac{v_i+v_j+1}{n(v_i+v_j)}=\frac{1}{n}+\frac{1}{n(v_i+v_j)} \geq \frac{1}{n}+\frac{1}{n(2v_{n-1}-1)}.$$
Applying Proposition~\ref{prop:perturb} with $\varepsilon=\frac{1}{n(2v_{n-1}-1)}$ and $L'=1/n+\varepsilon$ shows that
$$\ML(v_1, \ldots, v_n) \geq \frac{1}{n},$$
as desired.
\end{proof}

If we assume that the Spectrum Conjecture holds for $n-1$ runners, then we can take $f_n(v)$ to be linear in $v$ rather than quadratic in $v$ (at the cost of a dependence on $n$).  This improvement comes from the assumption that non-tight sets of $n-1$ speeds have maximum loneliness uniformly bounded away from $1/n$.

\begin{lemma}
Fix some $n \geq 3$, and assume that the Spectrum Conjecture holds for $n-1$ runners.  Let $v_1< \cdots < v_{n-1}$ be positive integers with $$\ML(v_1, \ldots, v_{n-1})=L>\frac{1}{n}.$$
Then we have
$$ML(v_1, \ldots, v_n) \geq \frac{1}{n}$$
whenever 
$$v_n \geq (2n-1)v_{n-1}.$$
\end{lemma}

\begin{proof}
By assumption, we have
$$L \geq \frac{2}{2n-1}.$$
Then applying Proposition~\ref{prop:perturb} with
$$\varepsilon=\frac{2}{2n-1}-\frac{1}{n}=\frac{1}{n(2n-1)}$$
and $L'=2/(2n-1)$ gives the desired result.
\end{proof}

In the next section, we turn our attention to the case $\ML(v_1, \ldots, v_{n-1}) \leq 1/n$.

\section{Accumulation points}\label{sec:accumulation}
We temporarily take a step back and consider the big picture.  For each positive integer $n$, define the set
$$\mathcal{S}(n)=\{\ML(v_1, \ldots, v_n): v_1, \ldots, v_n \text{ positive integers}\}$$
to consist of all maximum loneliness amounts achieved by sets of $n$ runners.  It is immediate that $\mathcal{S}(n) \subset \mathcal{S}(n+1)$ and that $\mathcal{S}(n) \subset (0, 1/2]$.  The Lonely Runner Conjecture asserts that $\mathcal{S}(n) \subset [1/(n+1),  1/2]$; the Loneliness Spectrum Conjecture, its refinement, asserts that $\mathcal{S}(n)$ is the union of $\{s/(ns+1): s \in \mathbb{N} \}$ and a subset of $[1/n, 1/2]$.

We define a real number $A$ to be an \emph{accumulation point} for $n$ runners if the set $\mathcal{S}(n)$ contains elements that are arbitrarily close to $A$.  On a more fine-grained view, we define $A$ to be a \emph{lower accumulation point} for $n$ runners if $\mathcal{S}(n)$ contains a sequence of elements approaching $A$ from below; we define $A$ to be an \emph{upper accumulation point} for $n$ runners if $\mathcal{S}(n)$ contains a sequence of elements approaching $A$ from above.  For instance, Theorem~\ref{thm:spectrum-construction} shows that $1/n$ is a lower accumulation point for $n$ runners.  We present the following pair of questions.
\begin{question}\label{question:lower-accumulation}
For $n \geq 2$, is the set of lower accumulation points for $n$ runners always precisely $\mathcal{S}(n-1)$?
\end{question}

\begin{question}\label{conj:upper-accumulation}
For $n \geq 2$, are there any upper accumulation points for $n$ runners, or are all accumulation points only lower accumulation points?
\end{question}

We make progress towards answering these questions, especially Question~\ref{question:lower-accumulation}.  For instance, the classification in Theorem~\ref{thm:2-runners} immediately answers both questions for $n=2$.
\begin{theorem}
The set of lower accumulation points for $2$ runners is precisely $\mathcal{S}(1)=\{1/2\}$.  There are no upper accumulation points for $2$ runners.
\end{theorem}

We now establish a refinement of Proposition~\ref{prop:perturb} in the case where we have additional information about the factors of $v_n$.

\begin{lemma}\label{lem:divisible}
Let $v_1<\cdots< v_{n-1}$ be positive integers ($n \geq 2$) with $$\ML(v_1, \ldots, v_{n-1})=L.$$
Let $t_1, \ldots, t_r$ be the times in the interval $[0,1)$ at which $\Abs{tv_i} \geq L$ for all $1 \leq i \leq n-1$.  For each $t_j$, let $\rho_j$ be the largest index such that $t_jv_{\rho_j}$ has remainder $L$ modulo $1$, and let $\lambda_j$ be the largest index such that $t_jv_{\lambda_j}$ has remainder $1-L$ modulo $1$.  Finally, define
$$\mu=\min_{1 \leq j \leq r} \{\rho_j, \lambda_j \},$$
where the minimum runs over all values of $j$.  Then
$$\ML(v_1,\ldots, v_n)=\frac{v_nL}{v_n+v_{\mu}}$$
whenever $v_n$ is a sufficiently large integer multiple of the lcm of the denominators of the $t_j$'s.
\end{lemma}

A few remarks are in order before we proceed to the proof.
\begin{itemize}
\item First, the fact that there are only finitely many ``equality times'' $t_j$ in $[0,1)$ is assured by Proposition~\ref{prop:times-to-check}.  It would also  suffice to consider $t_j$'s only in the interval $[0, 1/2]$.
\item Second, let us re-state in words what is expressed in symbols above: at each equality time $t_j$, we identify the fastest runner whose position in $\mathbb{R}/\mathbb{Z}$ is $L$ and the fastest runner whose position is $1-L$, then we let $v_{\mu}$ be the slowest of of these runners, across all the $t_j$'s.
\item Third, the statement of the theorem requires $v_n$ to be sufficiently large.  Carefully carrying all of the bounds through the proof shows that $$v_n>4v_{n-1}^3v_1$$
is sufficient; we make no attempts to optimize this quantity.
\end{itemize}

\begin{proof}
The divisibility condition on $v_n$ ensures that at every ``equality time'' $t_j$, we have $\Abs{t_jv_n}=0$, whence we conclude that
$$\ML(v_1, \ldots, v_n)<L$$
strictly.  We now record a few properties of the real function $$e(t)=\max_{1 \leq i \leq n-1} \Abs{tv_i}.$$
Note that $e(t)$ is continuous and piecewise linear and the slope of each linear segment has absolute value between $v_1$ and $v_{n-1}$.  Moreover, $e(t)$ has a local minimum at $t_0$ if and only if $e(t_0)=0$.

Fix some sufficiently small $\varepsilon>0$ (say, $\varepsilon \leq 1/(4v_{n-1}^3)$).  The above observations tell us that if $e(t) \geq L-\varepsilon$, then $t$ is a distance at most $\varepsilon/v_1$ from some $t'$ with $e(t')=L$.  In particular, if we want to identify all possible times $t$ (modulo $1$) with $e(t) \geq L-\varepsilon$, then it suffices to examine the closed $\varepsilon/v_1$-neighborhoods of the $t_j$'s.  Moreover, since $\varepsilon$ is sufficiently small, we have
$$e(t)=\Abs{tv_{\rho_j}} \quad \text{for all }t_j-\frac{\varepsilon}{v_1} \leq t \leq t_j$$
and
$$e(t)=\Abs{tv_{\lambda_j}} \quad \text{for all }t_j \leq t \leq t_j+\frac{\varepsilon}{v_1}.$$
This charcterization allows us to compute the maximum loneliness explicitly when we include the runner with speed $v_n$.

Fix some $t_j$, and recall that $\Abs{t_jv_n}=0$.  Write $t=t_j+\delta$.  As $\delta$ is increased from $0$, the quantity $\Abs{tv_n}$ increases from $0$ at a rate of $v_n$ and the quantity $\Abs{tv_{\lambda_j}}$ decreases from $L$ at a rate of $v_{\lambda_j}$.  We obtain equality $\Abs{tv_n}=\Abs{tv_{\lambda_j}}$ when $\delta=\delta_0=L/(v_n+v_{\lambda_j})$, at which point
$$\Abs{tv_n}=\Abs{tv_{\lambda_j}}=\frac{Lv_n}{v_n+v_{\lambda_j}}.$$
By choosing $v_n$ sufficiently large (say, $v_n\geq (Lv_1)/\varepsilon$), we guarantee that $\delta_0 \leq \varepsilon/v_1$.  Moreover, since the quantity $\Abs{tv_{\lambda_j}}$ is monotonically decreasing as $\delta$ increases, we conclude that $$\frac{Lv_n}{v_n+v_{\lambda_j}}$$
is the largest loneliness amount achieved for the speeds $v_1, \ldots, v_n$ for times $t_j \leq t \leq t_j+\varepsilon/v_1$.

The same reasoning shows that the largest loneliness amount achieved for times $t_j-\varepsilon/v_1 \leq t \leq t_j$ is
$$\frac{Lv_n}{v_n+v_{\rho_j}}.$$
Taking the maximum of all such quantities over $j$ gives that
$$\ML(v_1, \ldots, v_n)=\frac{v_nL}{v_n+v_{\mu}},$$
as desired.
\end{proof}

This lemma allows us to deduce that every element of $\mathcal{S}(n-1)$ is a lower accumulation point for $n$ runners.

\begin{theorem}\label{thm:lower-accumulation}
For every $n \geq 2$, the set of lower accumulation points for $n$ runners contains $\mathcal{S}(n-1)$.
\end{theorem}

\begin{proof}
Fix some $L \in \mathcal{S}(n-1)$.  Then there exist positive integers $v_1, \ldots, v_{n-1}$ such that $\ML(v_1, \ldots, v_{n-1})=L$.  By Lemma~\ref{lem:divisible}, there is an increasing sequence of values for $v_n$ such that the quantity $\ML(v_1, \ldots, v_n)$ approaches $L$ from below.
\end{proof}

In the absence of any straightforward constructions for upper accumulation points or other lower accumulation points, it is natural to suspect that all accumulation points for $n$ runners arise through the ``mechanism'' of Theorem~\ref{thm:lower-accumulation}.

An immediate corollary is that Loneliness Spectrum Conjecture (as well as its weakened version, Conjecture~\ref{conj:1-fast}) for $n$ runners implies the Lonely Runner Conjecture for $n-1$ runners.  We contrast this implication with the situation for the Lonely Runner Conjecture alone, where (to our knowledge), the statement for $n$ runners does not directly imply the statement for $n-1$ runners.  We record this observation in the following corollary.

\begin{corollary}\label{cor:implications}
Fix some. $n \geq 2$.  The Loneliness Spectrum Conjecture for $n$ runners implies Conjecture~\ref{conj:1-fast} for $n$ runners, which in turn implies the Lonely Runner Conjecture for $n-1$ runners.
\end{corollary}

\section{Re-formulating Conjecture~\ref{conj:1-fast}}\label{sec:reformulating}

We now present a generalization of Lemma~\ref{lem:divisible} to the scenario in which $v_n$ has a fixed residue with respect to certain moduli; the idea is substantively the same as the idea of Lemma~\ref{lem:divisible}, but the result is messier to state, so we introduce notation and an informal description before giving the precise statement and proof sketch.

As in Lemma~\ref{lem:divisible}, we fix positive integers $v_1, \ldots, v_{n-1}$ with $$\ML(v_1, \ldots, v_{n-1})=L,$$
and we let $t_1, \ldots, t_r$ be the ``equality times'' in $[0,1)$, i.e., the times for which $\Abs{tv_i} \geq L$ for all $1 \leq i \leq n-1$.  Let $D$ be the least common multiple of the denominators appearing in the reduced-fraction representations of the $t_j$'s.  Fix some integer $$-\frac{D}{2}<Q \leq \frac{D}{2}$$ such that $\Abs{t_j Q}<L$ for all $1 \leq j \leq r$.  (Call such a value of $Q$ \emph{admissible}.  We will look at values of $v_n$ that are equivalent to $Q$ modulo $D$.  Lemma~\ref{lem:divisible} is the special case $Q=0$.)  For each $j$, let $u_j$ denote the real number in $(-1/2, 1/2]$ that is equivalent to $t_jQ$ modulo $1$.  As in Lemma~\ref{lem:divisible}, for each $t_j$ we let $\rho_j$ be the largest index such that $t_jv_{\rho_j}$ has remainder $L$ modulo $1$, and we let $\lambda_j$ be the largest index such that $t_jv_{\lambda_j}$ has remainder $1-L$ modulo $1$.  Now, instead of taking a minimum over the $\rho_j$'s and $\lambda_j$'s, we take the following ``weighted minimum'': let $\mu$ be the $\rho_j$ or $\lambda_j$ that minimizes the quantity
$$v_{\rho_j}(L-u_j) \quad \text{or} \quad v_{\lambda_j}(L+u_j),$$
respectively.  We break ties between $\rho_j$'s in favor of larger $u_j$ (and arbitrarily beyond that point), and we break ties between $\lambda_j$'s in favor of smaller $u_j$ (and arbitrarily beyond that point).  We break ties between $\rho_{j'}$ and $\lambda_{j''}$ in favor of larger $u_j$ (and arbitrarily beyond that point).  So, keeping track of which $j$ our $\mu$ ``came from'' and whether it came from a $\rho$ or from a $\lambda$, we write either $\mu=\rho_k$ or $\mu=\lambda_k$.  We can finally state the lemma.

\begin{lemma}\label{lem:divisible-2}
Let $v_1, \ldots, v_{n-1}$ be positive integers ($n \geq 2$) with $$\ML(v_1, \ldots, v_{n-1})=L.$$
Define $t_1, \ldots, t_r$ and $D$ as above, and fix some admissible integer $$-\frac{D}{2}<Q \leq \frac{D}{2}.$$
Now define the $u_i$'s, $\rho_j$'s, and $\lambda_j$'s as above, along with the resulting $\mu=\rho_k$ or $\mu=\lambda{_k}$.  Then
$$
\ML(v_1, \ldots, v_n)=
\begin{cases}
L-\frac{v_{\rho_k}}{v_{\rho_k}+v_n}(L-u_k), &\text{if } \mu=\rho_k\\[10pt]
L-\frac{v_{\lambda_k}}{v_{\lambda_k}+v_n}(L+u_k), &\text{if } \mu=\lambda_k
\end{cases}
$$
whenever $v_n$ is a sufficiently large integer that is equivalent to $Q$ modulo $D$.
\end{lemma}

\begin{proof}[Proof (sketch).]
The choice of $Q$ guarantees that $\Abs{t_jv_n}=\Abs{t_jQ}<L$ for every $t_j$, which in turn implies that
$$\ML(v_1, \ldots, v_n)<L.$$
Moreover, note that at the time $t_j$, the runner with speed $v_n$ is at the position $u_j$ (which lies strictly between $-L$ and $L$).  The argument from the proof of Lemma~\ref{lem:divisible} shows that if $\Abs{tv_i} \geq L-\varepsilon$ for all $1 \leq i \leq n-1$, then $t$ is within $\varepsilon/v_1$ of some $t_j$, so we restrict our attention to these neighborhoods.  For times slightly larger than $t_j$, the greatest loneliness achieved by speeds $v_1, \ldots, v_n$ is precisely
$$L-\frac{v_{\lambda_j}}{v_{\lambda_j}+v_n}(L+u_j),$$
and for times slightly smaller than $t_j$, the greatest loneliness achieved is
$$L-\frac{v_{\rho_j}}{v_{\rho_j}+v_n}(L-u_j).$$
Taking a minimum over all such expressions (for sufficiently large $n$) gives the desired result.  (In other words, there is one expression that ``wins out'' for all sufficiently large $n$.)
\end{proof}

Unlike in Lemma~\ref{lem:divisible}, we do not compute an explicit characterization of $v_n$ ``sufficiently large''; the minimum in the last step of the proof complicates such a computation.  We can now give a necessary and sufficient condition to determine whether or not Conjecture~\ref{conj:1-fast} holds for speeds $v_1, \ldots, v_n$, where we fix $v_1, \ldots, v_{n-1}$.

\begin{theorem}\label{thm:condition}
Let $v_1<\cdots <v_{n-1}$ be positive integers ($n \geq 2$) with \linebreak $\ML(v_1, \ldots, v_{n-1})=L$.  Then the following are equivalent:
\begin{enumerate}[label=(\Roman*)]
\item For every sufficiently large integer $v_n$, we have either $$\ML(v_1, \ldots, v_n)=\frac{s}{ns+1} \quad \text{for some } s \in \mathbb{N} \quad \text{or} \quad \ML(v_1, \ldots, v_n) \geq \frac{1}{n}.$$
\item One of the following holds:
\begin{enumerate}
\item $L>1/n$.
\item $L=1/n$; and (in the notation of Lemma~\ref{lem:divisible-2}) for each admissible residue $Q$ we have, when $\mu=\rho_k$ (respectively,  $\mu=\lambda_k$), both the equality 
$$Q=-nv_{\rho_k}u_k \quad \text{(respectively, } Q=nv_{\lambda_k}u_k \text{)}$$
and the property that $$\frac{D}{nv_{\rho_k}(1-nu_k)} \quad \text{(respectively, } \frac{D}{nv_{\lambda_k}(1+nu_k)} \text{)}$$
is an integer.
\end{enumerate}
\end{enumerate}
\end{theorem}

\begin{proof}
It is clear from Lemma~\ref{lem:divisible-2} that if $L>1/n$, then $\ML(v_1, \ldots, v_n)>1/n$ for all sufficiently large $v_n$; the reverse implication follows from Lemma~\ref{lem:divisible}.  It is also clear from Lemma~\ref{lem:divisible} that if $L<1/n$, then (I) does not hold.  So it remains to consider the case where $L=1/n$.  For each integer $-D/2<Q \leq D/2$, we consider sufficiently large values of $n$ with (fixed) residue $Q$ modulo $D$.  If there is any $t_j$ with $\Abs{t_jQ} \geq 1/n$, then the speeds $v_1, \ldots, v_n$ achieve a loneliness of $1/n$ at that time, whence we conclude that $\ML(v_1, \ldots, v_n)=1/n$.  So, as in Lemma~\ref{lem:divisible-2} we restrict our attention to admissible values of $Q$.

Fix some such $Q$, and write $v_n=mD+Q$.  Suppose that in Lemma~\ref{lem:divisible-2}, we have $\mu=\rho_k$.  Then for sufficiently large $m$ (i.e., sufficiently large $v_n$), we have
$$\ML(v_1, \ldots, v_n)=\frac{1}{n}-\frac{v_{\rho_k}}{v_{\rho_k}+v_n} \left(\frac{1}{n}-u_k\right).$$
Suppose this quantity equals $s/(ns+1)$ for some $s \in \mathbb{N}$, i.e.,
$$\frac{v_{\rho_k}}{v_{\rho_k}+v_n} \left(\frac{1}{n}-u_k \right)=\frac{1}{n(ns+1)}.$$
Substituting for $v_n$ and rearranging gives
$$\frac{mD+Q+v_{\rho_k}}{v_{\rho_k}(1-nu_k)}=ns+1.$$
Each (sufficiently large) $m$ can have a corresponding $s$ only if incrementing $m$ increments the left-hand side by an integer multiple of $n$, i.e.,
$$\frac{D}{nv_{\rho_k}(1-nu_k)}$$
is an integer.
Moreover, we then see that we also require
$$\frac{Q+v_{\rho_k}}{v_{\rho_k}(1-nu_k)}=1,$$
or, equivalently,
$$Q=-nv_{\rho_k}u_k.$$
These are precisely the two conditions of (II.a).  The same argument for $\mu=\lambda_k$ gives the analogous pair of conditions in the statement of the lemma.  Finally, to achieve a uniform bound (over choices of $Q$) on how large $v_n$ must be, we simple take the maximum of the bounds obtained for the various $Q$'s.
\end{proof}

This theorem reduces Conjecture~\ref{conj:1-fast} for $n$ runners to an explicit computation once we know all of the tight speed sets for $n-1$ runners; moreover, this computation is finite if there are only finitely many tight sets of speeds (up to scaling).  Recall from Corollary~\ref{cor:implications} that if the Lonely Runner Conjecture does not hold for $n-1$ runners, then Conjecture~\ref{conj:1-fast} does not hold for $n$ runners.

\begin{theorem}\label{thm:reduction-to-computation}
Fix some $n \geq 4$.  Suppose the Lonely Runner Conjecture holds for $n-1$ runners and we know all of the tight speed sets for $n-1$ runners.  Then Conjecture~\ref{conj:1-fast} holds for $n$ runners if and only if every such tight set of speeds $v_1<\cdots <v_{n-1}$ with $\gcd(v_1, \ldots, v_{n-1})=1$ satisfies the conditions in (II.b) of Theorem~\ref{thm:condition} (which can be checked with an explicit finite computation).
\end{theorem}

\begin{proof}
First, recall that Conjecture~\ref{conj:1-fast} asks for a function $f_n:\mathbb{N} \to \mathbb{N}$ such that the set of speeds $\ML(v_1, \ldots, v_n)$ has certain properties whenever $v_n>f_n(v_{n-1})$; this is \emph{uniform} bound on $v_n$ in terms of $v_{n-1}$.  Since there are only finitely many sets of positive speeds $v_1<\cdots<v_{n-1}$ for each value of $v_{n-1}$, however, it suffices to consider ``sufficiently large'' $v_n$ for each set $v_1, \ldots, v_{n-1}$ separately and then take $f_n(v_{n-1})$ to be the maximum of the bounds obtained.

Suppose we have verified the conditions in (II.b) of Theorem~\ref{thm:condition} for every tight set of speeds $v_1<\cdots <v_{n-1}$ with $\gcd(v_1, \ldots, v_{n-1})=1$.  Then we claim that (I) is also satisfied for every set of tight speeds.  Indeed, let $v'_1<\cdots < v'_{n-1}$ be positive integers with $\ML(v'_1, \ldots, v'_{n-1})=1/n$.  Then let $g=\gcd(v'_1, \ldots, v'_{n-1})$, and write $v'_i=gv_i$ for $1 \leq i \leq n-1$, where we know that (II.b) is satisfied for the speeds $v_1, \ldots, v_{n-1}$.  If $g=1$, then we are done, so consider $g \geq 2$.  The equality times for $v'_1, \ldots, v'_{n-1}$ are precisely the times of the form $$\frac{t_j+h}{g},$$ where $t_j$ is an equality time for $v_1, \ldots, v_{n-1}$ and $h$ is an integer.  Thus, we have $D'=gD$, where $D'$ (respecticely, $D$) is the lcm of the equality times in $[0,1)$ for the speeds $v'_1, \ldots, v'_{n-1}$ (respectively, $v_1, \ldots, v_{n-1}$).  We now consider various admissible residues $Q'$ (modulo $D'$).  Each admissible $Q'$ must be a multiple of $g$: otherwise, we could add time increments of $1/g$ (pre-jump) to find an equality time $t'_j$ for $v'_1, \ldots, v'_{n-1}$ at which $$\Abs{t'_jQ'}\geq \frac{1}{2}-\frac{1}{2g} \geq \frac{1}{4} \geq \frac{1}{n}.$$
So $Q'$ is a multiple of $g$, and any $v'_n$ that is equivalent to $Q'$ modulo $D'$ can be written as $v'_n=gv_n$.  But then
$$\ML(v'_1, \ldots, v'_n)=\ML(v_1, \ldots, v_n),$$
and we know that the quantity on the right-hand side satisfies condition (I) of Theorem~\ref{thm:condition}.  So we conclude that it suffices to check tight instances for $n-1$ runners where the speeds do not all share a common factor.
\end{proof}

Another point of interest of this theorem is that it provides a potential way to refute the Spectrum Conjecture.

We now apply Theorem~\ref{thm:condition} to the tight set of speeds $1, \ldots, n-1$ and then use Theorem~\ref{thm:reduction-to-computation} to resolve Conjecture~\ref{conj:1-fast} for $4$ moving runners.  The computation is straightforward.

\begin{prop}\label{prop:computation-for-consecutive-speeds}
Consider the tight set of speeds $v_1=1, v_2=2, \ldots, v_{n-1}=n-1$ ($n \geq 4$).  This set of speeds satisfies the conditions of (II.b) in Theorem~\ref{thm:condition}.
\end{prop}

\begin{proof}
The equality times for $v_1, \ldots, v_{n-1}$ are precisely the times of the form $$\frac{m}{n},$$
where $m$ is relatively prime to $n$, so $D=n$.  The only admissible value of $Q$ is $0$ because otherwise we would have $\Abs{tQ} \geq 1/n$ at the equality time $t=1/n$.  Since $Q=0$, we are in fact in the setting of Lemma~\ref{lem:divisible}, so we simply have to find the smallest $\rho_j$ or $\lambda_j$.  We get $\rho_1=1$ at the time $t_1=1/n$, and this is the smallest possible.  Since $Q=0$ implies that $u_1=0$, the first condition of (II.b) is immediately satisfied.  For the second condition, it suffices to observe that
$$\frac{D}{nv_{\rho_1}(1-nu_1)}=\frac{n}{n(1)(1-0)}=1$$
is an integer.
\end{proof}
Since we stayed in the Lemma~\ref{lem:divisible} ``special subcase'' of Lemma~\ref{lem:divisible-2}, we are justified in using the explicit bounds on $v_n$ ``sufficiently large'' from the former lemma.  We remark that it would be interesting to carry out these computations for the other tight sets of speeds discussed in Goddyn and Wong \cite{goddyn}.

Recall from Corollary~\ref{cor:3-runners} that the only tight speed set for $3$ runners is, up to scaling, $1,2,3$.  It then follows from the preceding discussion that Conjecture~\ref{conj:1-fast} holds for $n=4$ (and the function $f_4$ can be taken to be quartic in $v_{n-1}$).

\begin{corollary}
There exists a function $f_4: \mathbb{N} \to \mathbb{N}$ such that for any positive integers $v_1<v_2<v_3<v_4$ with $v_4>f_4(v_3)$, we have either $$\ML(v_1, v_2, v_3, v_4)=\frac{s}{4s+1} \quad \text{for some } s \in \mathbb{N} \quad \text{or} \quad \ML(v_1,v_2,v_3,v_4)\geq \frac{1}{4}.$$
\end{corollary}

We can carry out the same program for $n=6$ by making use of Bohman, Holzman, and Kleitman's determination \cite{bohman} of all of the tight sets of speeds for $5$ moving runners.

\begin{theorem}[Bohman, Holzman, and Kleitman \cite{bohman}]\label{thm:tight-5}
The Lonely Runner \linebreak Conjecture holds for $n=5$.  Moreover, if $v_1, \ldots, v_5$ are positive integers with \linebreak $\gcd(v_1, \ldots, v_5)=1$ and $\ML(v_1, \ldots, v_5)=1/6$, then $v_1, \ldots, v_5$ are (in some order) either $1,2,3,4,5$, or $1,3,4,5,9$.
\end{theorem}

\begin{corollary}
There exists a function $f_6: \mathbb{N} \to \mathbb{N}$ such that for any positive integers $v_1<\cdots<v_6$ with $v_6>f_6(v_5)$, we have either $$\ML(v_1, \ldots, v_6)=\frac{s}{6s+1} \quad \text{for some } s \in \mathbb{N} \quad \text{or} \quad \ML(v_1, \ldots, v_6)\geq \frac{1}{6}.$$
\end{corollary}

\begin{proof}
The set of tight speeds $1,2,3,4,5$ is handled by Proposition~\ref{prop:computation-for-consecutive-speeds}.  For $1,3,4,5,9$, the only equality times in $[0,1)$ are $1/6$ and $5/6$.  So $D=6$, and it is easy to check that only $Q=0$ is admissible.  As in the proof of Proposition~\ref{prop:computation-for-consecutive-speeds}, we find ourselves in the setting of Lemma~\ref{lem:divisible}, where we get the ``best possible'' value $v_{{\rho_1}}=1$ at the time $t_1=1/n$, and the conditions of (II.b) are satisfied in the same way.
\end{proof}

It is curious that the literature seems not to contain a characterization of all tight sets of speeds with $4$ moving runners.  It is widely known (see, e.g., \cite{goddyn}) that $1,3,4,7$ is a tight speed set in addition to the trivial $1,2,3,4$.  The reader may easily verify that the conditions of (II.b) are satisfied for $1,3,4,7$.  It appears likely that there are no other tight sets of speeds (up to scaling), in which case we would also be able to confirm Conjecture~\ref{conj:1-fast} for $n=5$.

\section{Speeds in a short arithmetic progression}\label{sec:short-progression}

Tao \cite{tao} obtained the following result for the Lonely Runner Conjecture in the case where all speeds are small.

\begin{theorem}[Tao \cite{tao}]\label{thm:tao}
If $v_1, \ldots, v_n$ are positive integers ($n \geq 1$) all less than or equal to $1.2n$, then $\ML(v_1, \ldots, v_n) \geq 1/(n+1)$.
\end{theorem}

We make a few trivial observations about how this line of inquiry relates to our refined conjecture about the loneliness spectrum.  We begin by confirming the Spectrum Conjecture for speeds up to $1.2n$.
\begin{prop}\label{prop:1.2n}
If $v_1, \ldots, v_n$ are positive integers ($n \geq 2$) all less than or equal to $1.2n$, then one of the following holds:
\begin{itemize}
\item $\ML(v_1, \ldots, v_n)=1/(n+1)$ or $\ML(v_1, \ldots, v_n)=2/(2n+1)$.
\item $\ML(v_1, \ldots, v_n) \geq 1/n$.
\end{itemize}
\end{prop}

\begin{proof}
By Theorem~\ref{thm:tao}, we know that $\ML(v_1, \ldots, v_n) \geq 1/(n+1)$.  We also know from Proposition~\ref{prop:times-to-check} that $\ML(v_1, \ldots, v_n)$ is a rational number with denominator at most $2.4n-1$ when expressed as a reduced fraction.  The only rational number between $1/(n+1)$ and $1/n$ with denominator at most $2.4n-1$ is $2/(2n+1)$.
\end{proof}
In fact, this argument shows that the Lonely Runner Conjecture and the Spectrum Conjecture are equivalent for speeds up to $1.5n+1$.

\begin{prop}
If $v_1, \ldots, v_n$ are positive integers ($n \geq 2$) all less than or equal to $1.5n+1$, then one of the following holds:
\begin{itemize}
\item $\ML(v_1, \ldots, v_n)<1/(n+1)$.
\item $\ML(v_1, \ldots, v_n)=s/(ns+1)$ for some integer $1 \leq s \leq 3$.
\item $\ML(v_1, \ldots, v_n) \geq 1/n$.
\end{itemize}
\end{prop}
\begin{proof}
We know that $\ML(v_1, \ldots, v_n)$ is a rational number with denominator at most $3n+1$.
\end{proof}

It would be interesting to extend results in this direction.  For instance, the first step would be to establish something along the lines of the following.

\begin{conjecture}\label{conj:1.501n}
If $v_1, \ldots, v_n$ are positive integers ($n \geq 2$) all less than or equal to $1.501n$, then $\ML(v_1, \ldots, v_n) \neq 3/(3n+2)$.
\end{conjecture}

\section{Interlude: random runners and lacunary runners}\label{sec:interlude}

In this section, we briefly generalize two existing results from the literature on the Lonely Runner Problem.

First, we discuss runners with ``random'' speeds.  Czerwi\'{n}ski \cite{czerwinski} used Fourier-analytic methods to prove the following result, which says that almost all sets of speeds have maximum loneliness arbitrarily close to $1/2$.

\begin{theorem}[Czerwi\'{n}ski \cite{czerwinski}]\label{thm:random}
Fix a positive integer $n$, and fix any $\varepsilon>0$.  If a subset $\{v_1, \ldots, v_n\}$ is chosen uniformly at random from all $n$-element subsets of $\{1, 2, \ldots, N\}$, then the probability that
$$\ML(v_1, \ldots, v_n) \geq \frac{1}{2}-\varepsilon$$
tends to $1$ as $N$ approaches infinity.
\end{theorem}

In fact, Czerwi\'{n}ski's argument gives the quantitative bound
$$\mathbb{P}\left[\ML(v_1, \ldots, v_n) \geq \frac{1}{2}-\varepsilon \right]>1-\frac{(2N^{\frac{1}{n+1}}+1)^n}{N-n}$$
whenever $N$ is large enough to satisfy
$$\sqrt{\frac{n^3 3^{n-1}}{2^{n+1}\varepsilon^{2n}}}+1<N^{\frac{1}{n+1}}.$$
(Czerwi\'{n}ski states this quantitative version only for prime $N$.  Also, Alon \cite{alon} later improved this bound.)  We mention the parts of Czerwi\'{n}ski's proof that we need in order to generalize his result.  For a prime $p$ and an element $x \in \mathbb{F}_p$, define $\Abs{x}_p$ to be the minimum distance from a representative in $\mathbb{Z}$ of $x$ to a multiple of $p$.  Then, given a positive integer $\ell$, say that a subset $\{v_1, \ldots, v_n\} \subseteq \mathbb{F}_p$ is \emph{$\ell$-independent} if every solution over $\mathbb{F}_p$ to
$$c_1v_1+\cdots+c_nv_n=0$$
has
$$\Abs{c_1}_p+\cdots+\Abs{c_n}_p > \ell.$$
Czerwi\'{n}ski's main tool is the following lemma.

\begin{lemma}[Czerwi\'{n}ski \cite{czerwinski}]\label{lem:l-independent}
Fix a positive integer $n$, and fix any $\varepsilon>0$.  Let $v_1<\cdots<v_n< p$ be positive integers.  If the image of $\{v_1, \ldots, v_n\}$ in $\mathbb{F}_p$ is $\ell$-independent with
$$\ell> \sqrt{\frac{n^3 3^{n-1}}{2^{n+1}\varepsilon^{2n}}},$$
then $\ML(v_1, \ldots, v_n) \geq 1/2-\varepsilon$.
\end{lemma}

Czerwi\'{n}ski deduces Theorem~\ref{thm:random} from this lemma by upper-bounding the number of $\ell$-dependent subsets of size $n$ in $\mathbb{F}_p$.  Our contribution consists of the observation that one can apply Lemma~\ref{lem:l-independent} to sets of speeds that come from $N$-element sets other than $\{1, \ldots, N\}$.  We follow Czerwi\'{n}ski's argument very closely.

\begin{theorem}\label{thm:random-2}
Fix a positive integer $n$, and fix any $\varepsilon>0$.  If a subset $\{v_1, \ldots, v_n\}$ is chosen uniformly at random from all $n$-element subsets of some $S \subset \mathbb{N}$ of cardinality $N$, then
$$\mathbb{P} \left[\ML(v_1, \ldots, v_n) \geq \frac{1}{2}-\varepsilon \right]>1-\frac{(2N^{\frac{1}{n+1}}+1)^n}{N-n}$$
whenever
$$\sqrt{\frac{n^3 3^{n-1}}{2^{n+1}\varepsilon^{2n}}}+1<N^{\frac{1}{n+1}}.$$
\end{theorem}

\begin{proof}
Let $p$ be a prime that is larger than all of the elements of $S$.  Then there is an integer $\ell$ such that
$$\sqrt{\frac{n^3 3^{n-1}}{2^{n+1}\varepsilon^{2n}}}<\ell<N^{\frac{1}{n+1}}.$$
The number of $n$-element subsets of $S$ with $\ell$-dependent image in $\mathbb{F}_p$ is at most
$$(2\ell+1)^n \binom{N-1}{n-1}.$$
To see that this quantity is an upper bound on the number of subsets with ``bad'' linear dependences, note that the first term counts the ways to pick $c_1, \ldots, c_n$, each between $-\ell$ and $\ell$, and the second term is an upper bound on the number of ways to choose $v_1<\cdots<v_{n-1}$ from $S$; at this point, there is at most $1$ choice for $v_n$ that satisfies $c_1v_1+\cdots+c_nv_n=0$.  The probability that the image of $S$ is $\ell$-dependent is thus at most
$$\frac{(2\ell+1)^n \binom{N-1}{n-1}}{\binom{N-1}{n}}=\frac{n(2\ell+1)^n}{N-n}<\frac{(2N^{\frac{1}{n+1}}+1)^n}{N-n},$$
then the result follows from Lemma~\ref{lem:l-independent}.
\end{proof}

We mention one further aspect of Czerwi\'{n}ski's method that may be useful for future investigations.  He considers only times of the form $t=m/p$, where $p$ is a fixed prime that is larger than all of the speeds, so that he may bring in tools from Fourier analysis.  (This strategy is not unique to Czerwi\'{n}ski.)  In general, the aim is to find some $m$ such that every $\Abs{mv_i}_p$ is large, i.e., $\Abs{tv_i}$ is large.

This strategy at first glance seems ill-matched for proving results about the discrete part of the loneliness spectrum because checking times of the form $m/p$ can give us only a \emph{lower} bound on the maximum loneliness.  But recall from Proposition~\ref{prop:times-to-check} that the maximum loneliness must be a rational with bounded denominator (in terms of the speeds), and, in particular, any two such ``candidate'' times are uniformly bounded apart by $1/(4v_n^2)$, where $v_n$ is the speed of the fastest runner.  The discussion of the function $e(t)$ in the proof of Lemma~\ref{lem:divisible} shows that loneliness does not vary very quickly with time.  So for $p$ sufficiently large, knowing the greatest loneliness attained at a time of the form $m/p$ allows us to reconstruct the actual maximum loneliness.  This idea could allow us to apply Fourier-analytic methods to the intervals of maximum loneliness that are forbidden by the Spectrum Conjecture.
\\

Second, we discuss sets of speeds that form lacunary sequences.  Previous results of this flavor usually have the following form: if $v_1, \ldots, v_n$ satisfy $$\frac{v_i}{v_{i-1}} \geq c(n)$$ for all $2 \leq i \leq n$ and some specified $c(n)$, then $\ML(v_1, \ldots, v_n) \geq 1/(n+1)$.  One drawback to this approach is that the quantity $1/(n+1)$ usually does not play a particularly special role; at the cost of a worse constant, one could use, say, $1/n$ instead.  This plasticity suggests that such an approach is unlikely to lead to a full resolution of the Lonely Runner Problem.  We draw attention to this feature in an almost-state-of-the-art result of Dubickas \cite{dubickas}.  (Czerwi\'{n}ski \cite{czerwinski2} has recently obtained a slight refinement.)  Dubickas' main result (which holds for general real speeds) is as follows.

\begin{theorem}[Dubickas \cite{dubickas}]\label{thm:lacunary}
Fix some $n \geq 32$, and let $$h=\left\lfloor \frac{n+1}{12e} \right\rfloor.$$
If $v_1<\cdots<v_n$ are positive real numbers satisfying $v_{i+h} \geq (n+1) v_i$ for every $1 \leq i \leq n-h$, then we have $\ML(v_1, \ldots, v_n) > 1/(n+1)$.  (Here, $e$ is the base of the natural logarithm.)
\end{theorem}

This theorem gives the following immediate corollary.

\begin{corollary}[Dubickas \cite{dubickas}]\label{cor:lacunary}
For every real number $\kappa>12e$, there exists a threshold integer $N(\kappa)$ such that the following holds: if $n \geq N(\kappa)$ and the positive real numbers $v_1<\cdots <v_n$ satisfy $$\frac{v_i}{v_{i-1}} \geq 1+\frac{\kappa \log n}{n}$$
for all $2 \leq i \leq n$, then $\ML(v_1, \ldots, v_n) > 1/(n+1)$.  (Here, $\log$ denotes the natural logarithm.)
\end{corollary}

Dubickas gives a slightly more involved argument that lets one replace $12e$ with $8e$.  His proof makes use of the following lemma (due to the same author \cite{dubickas2}), which in turn relies on the Lov\'{a}sz Local Lemma.  See also a similar technique of Peres and Schlag \cite{peres}.

\begin{lemma}[Dubickas \cite{dubickas2}]\label{lem:local}
Let $h$ be a positive integer, and fix any real number $w \geq 4eh$.  If $v_1<\cdots <v_n$ are positive real numbers satisfying $$v_{i+h} \geq w v_i$$
for all $1 \leq i \leq n-h$, then there a real number $t$ such that $$\Abs{tv_i} > \frac{1}{8eh}-\frac{1}{2w}$$
for all $1 \leq i \leq n$.
\end{lemma}

We use this lemma to derive a generalization of Theorem~\ref{thm:lacunary}.  We do not try to optimize constants.

\begin{theorem}\label{thm:lacunary2}
Fix some integer $c$ and some real number $\alpha>0$.  Then for all sufficiently large $n$, the following holds with $$h=\left\lfloor \frac{n+c}{12e\alpha} \right\rfloor \quad \text{and} \quad w=\frac{n+c}{\alpha}:$$
 if $v_1<\cdots<v_n$ are positive real numbers satisfying $v_{i+h} \geq w v_i$ for every $1 \leq i \leq n-h$, then we have $$\ML(v_1, \ldots, v_n) > \frac{\alpha}{n+c}.$$
\end{theorem}

\begin{proof}
Lemma~\ref{lem:local} with $h$ and $w$ as stated tells us that there exists a real number $t$ such that for all $1 \leq i \leq n$, we have
$$\Abs{tv_i} > \frac{1}{8eh}-\frac{1}{2w} \geq \frac{12e\alpha}{8e(n+c)}-\frac{\alpha}{2(n+c)}=\frac{\alpha}{n+c},$$
as desired.
\end{proof}

And, as before, we obtain a corollary on lacunary sequences.

\begin{corollary}
For every real number $\kappa>12e\alpha$, there exists a threshold integer $N(\kappa)$ such that the following holds: if $n \geq N(\kappa)$ and the positive real numbers $v_1<\cdots <v_n$ satisfy $$\frac{v_i}{v_{i-1}} \geq 1+\frac{\kappa \log n}{n}$$
for all $2 \leq i \leq n$, then $\ML(v_1, \ldots, v_n) > \alpha/(n+c)$.
\end{corollary}

Of course, the case $\alpha=1$, $c=0$ is of particular interest for the Spectrum Conjecture.

\section{Concluding remarks} \label{sec:conclusion}

Over the course of this paper, we have raised a number of questions, problems, and ideas that could be fruitful starting points for future research on the loneliness spectrum.  For the sake of convenience, we gather many of these points, and a few others, here.
\begin{itemize}
\item Prove the Spectrum Conjecture for $n=4$.  Is it feasible to extend the techniques that we used for the $n=3$ case?
\item Determine whether or not $1,2,3,4$ and $1,3,4,7$ are (up to scaling) the only tight speed sets for $4$ runners.
\item Determine the set of accumulation points for $n$ runners.  Are there any upper accumulation points?
\item In all of our applications of Theorem~\ref{thm:condition}, only $Q=0$ is admissible.  Is this the case for all tight speed sets?
\item Check the conditions in (II.b) of Theorem~\ref{thm:condition} for the families of tight speed sets in Goddyn and Wong \cite{goddyn}.
\item Establish the Spectrum Conjecture for positive integer speeds up to $\beta n$, for some $\beta>1.5$.
\item Is it possible to leverage Fourier-analytic techniques for proving results on the loneliness spectrum?
\item What consequences would the Spectrum Conjecture have for related areas, such as chromatic numbers of distance graphs?
\end{itemize}

\section*{Acknowledgements}
This paper was written in partial fulfillment of the Mathematics major senior requirement at Yale University.  I would like to thank my adviser Stefan Steinerberger for his support, encouragement, and feedback in both the brainstorming and writing stages of this project.  I developed many of this paper's ideas while running long distances, usually not around a track.

\end{document}